\documentclass{amsart}

\title[Topological $G$-manifolds III]{Countable approximation of topological $G$-manifolds, III:  arbitrary Lie groups $G$}

\author[Q Khan]{Qayum Khan}
\address{Department of Mathematics \hfill Indiana University \hfill Bloomington IN 47405 USA}
\email{qkhan@indiana.edu}

\usepackage[utf8]{inputenc}
\usepackage[russian,english]{babel}

\usepackage{amsthm,amsmath,amssymb,eucal,mathrsfs,verbatim}
\usepackage{url}

\usepackage{xcolor}
\definecolor{dark-red}{rgb}{0.4,0.15,0.15}
\definecolor{dark-blue}{rgb}{0.15,0.15,0.4}
\definecolor{medium-blue}{rgb}{0,0,0.5}
\usepackage[colorlinks, linkcolor={dark-red}, citecolor={dark-blue}, urlcolor={medium-blue}]{hyperref} 

\newtheorem{thm}{Theorem}[section]

\newtheorem{cor}[thm]{Corollary}
\newtheorem{lem}[thm]{Lemma}
\newtheorem{prop}[thm]{Proposition}

\theoremstyle{definition}
\newtheorem{defn}[thm]{Definition}

\newtheorem{rem}[thm]{Remark}
\newtheorem{exm}[thm]{Example}

\numberwithin{equation}{section}

\DeclareMathAlphabet{\matheurm}{U}{eur}{m}{n}

\newcommand{\M}{\mathbb{M}}
\newcommand{\N}{\mathbb{N}}
\newcommand{\Q}{\mathbb{Q}}
\newcommand{\R}{\mathbb{R}}

\newcommand{\Z}{\mathbb{Z}}

\renewcommand{\G}{\Gamma}

\newcommand{\cC}{\mathcal{C}}
\newcommand{\cD}{\mathcal{D}}

\newcommand{\cL}{\mathcal{L}}

\newcommand{\cM}{\mathcal{M}}

\newcommand{\cO}{\mathcal{O}}
\newcommand{\cP}{\mathcal{P}}

\newcommand{\cS}{\mathcal{S}}
\newcommand{\cT}{\mathcal{T}}

\newcommand{\cpt}{\matheurm{cpt}}

\newcommand{\fin}{\matheurm{fin}}

\newcommand{\Cpt}{\mathrm{Kpt}}

\newcommand{\Homeo}{\mathrm{Homeo}}

\newcommand{\id}{\mathrm{id}}

\newcommand{\Map}{\mathrm{Map}}

\newcommand{\Or}{\mathrm{Or}}

\newcommand{\co}{\mathrm{co}}

\newcommand{\eps}{\varepsilon}

\newcommand{\longra}{\longrightarrow}

\newcommand{\x}{\times}

\newcommand{\ol}[1]{\overline{#1}}

\newcommand{\inv}{^{-1}}
\newcommand{\Sec}{\mathrm{Sec}}

\newcommand{\Int}{\mathrm{Int}}

\newcommand{\cco}{\ol{\mathrm{co}}}

\newcommand{\card}{\mathrm{card}}

\begin{document}

\dedicatory{This paper is dedicated to the memory of Jan W Jaworowski (1928--2013),\\ with whom the author had brief cordial interactions at Indiana University.}

\begin{abstract}
The Hilbert--Smith conjecture states, for any connected topological manifold $M$, any locally compact subgroup of $\Homeo(M)$ is a Lie group.
We generalize basic results of Segal--Kosniowski--tomDieck (\ref{prop:tomDieck}), James--Segal (\ref{cor:JamesSegal}), G~Bredon (\ref{cor:transtoral}), Jaworowski--Antonyan \emph{et al.}~(\ref{thm:criterion}), and E~Elfving (\ref{cor:main}).
The last is our main result: for any Lie group $G$, any Palais-proper topological $G$-manifold has the $G$-homotopy type of a countable proper $G$-CW complex.
Along the way, we verify an $n$-classifying space for principal $G$-bundles (\ref{thm:classify}).
\end{abstract}

\maketitle

\section*{Historical introduction and statement of results}

Interpreting an ambitious desire (1854) from  B~Riemann's habilitation, E~Betti formalized the extrinsic definition (1871) of \emph{differentiable ($C^1$) manifold,} as a nonsingular differentiable zero-locus, via the implicit function theorem of Cauchy--Dini.  Later, H~Weyl gave the intrinsic definition (1913) of \emph{topological ($C^0$) manifold,} as a locally euclidean separable metric space, the latter from M~Fr\'{e}chet's thesis (1906).

Kuratowski's thesis provided the axiomatic definition (1922) of topological space.
Besides manifolds, we shall consider three well-known classes of topological spaces.
Firstly, O~Schreier introduced the notion (1925) of a $T_0$ (so $T_{3.5}$) \emph{topological group:} a Hausdorff space whose underlying set is a group such that the division function $(x,y) \longmapsto x y\inv$ is continuous.
Examples are Lie groups and the $p$-adics.
This concept is motivated by a continuous group of transformations, that is, a group action.
Secondly, the \emph{absolute neighborhood extensor (ANE)} was introduced by K~Borsuk (1932): a normal Hausdorff space satisfying that any continuous function to it from any closed subset of any metrizable space extends to a neighborhood.
Examples are simplicial polytopes with either metric or weak topology, as well as topological manifolds.
Thirdly, J\,H\,C~Whitehead invented the \emph{CW complex} (1949): a cell complex whose closure of an open cell consists of only finitely many open cells and a subset of the space is open means its intersection with any open cell is open.
Of course by design, simplicial polytopes with weak topology are CW complexes.

Allowing for nonlinear (so noncompact) formulation, a \emph{Lie group} (1873) is a real-analytic manifold (the transition functions are $C^\omega$) equipped with a group structure such that $(x,y) \longmapsto x y\inv$ is real-analytic.
Hilbert's fifth problem (1900) has a first interpretation proven by A~Gleason (1950) along with Montgomery--Zippin (1952): the Lie groups are exactly the topological groups that are topological manifolds.
A second interpretation of the fifth problem generalizes it to effective group actions and not merely group translations; it is called the Hilbert--Smith conjecture (1941): any locally compact subgroup $G$ of the homeomorphism group (equipped with the compact-open topology) of a connected topological manifold is Lie.
Limited by this conjecture, our main results are for Lie groups $G$.
Namely, J~Jaworowski's $G$-ANE criterion (1981) for compact Lie groups, generalized by S~Antonyan \emph{et al.} to linear Lie groups (2017), is further generalized to all Lie $G$ (\ref{thm:criterion}).
However, when possible, we consider locally compact $G$, such as in our generalization of Segal--Kosniowski--tomDieck's correspondence (1987) between $G$-maps and sections (\ref{prop:tomDieck}).
Our main theorem (\ref{thm:main}) is that any second-countable \emph{$\Z$-cohomology manifold,} in the sense of A~Borel's Smith-theory seminar (1960), equipped with a Palais-proper action of a Lie group $G$ is a $G$-ANR if and only if the fixed set of each subgroup is locally contractible.
Here, one can define ($G$-)ANR as ($G$-)metrizable ($G$-)ANE.
Along the way, we correct gaps and generalize to all Lie $G$ a criterion of James--Segal (1980): a $G$-map between suitable proper $G$-ANRs is a $G$-homotopy equivalence if and only if it restricts to homotopy equivalences between the $H$-fixed sets (\ref{cor:JamesSegal}).

Algebraic topologists prefer CW complexes over ANEs because open cells are euclidean spaces, quotients of cellular maps have an induced CW structure, and one can induct on dimension.
Equipped with the cell-wise product, CW complexes and cellular maps form a subcategory of the compactly generated Hausdorff spaces.
For any topological space $X$, S~Eilenberg introduced (1944) the singular complex $S(X)$, and J~Giever proved (1950) that the evaluation map $|S(X)| \longra X$ from its realization polytope equipped with weak topology is a weak homotopy equivalence.

Despite being functorial, $S(X)$ has an \emph{excess number} of simplices: $\mathfrak{c}$ if $\card(X) \leqslant \mathfrak{c}$ else $(\card\,X)^{\aleph_0}$, where $\mathfrak{c} := 2^{\aleph_0} = \card(\R)$ is the cardinality of the continuum.
Nonetheless, O~Hanner showed (1951) any separable metrizable ANR is homotopy equivalent to a \emph{countable} (via E~Lindel\"of's lemma, 1904) locally finite polytope.
For Lie $G$, by observations of the author \cite[Proof~3.1]{Khan_linearLie} on Antonyan--Elfving (2009), any separable $G$-ANR has the $G$-homotopy type of a countable proper $G$-CW complex.
Combined with the above main theorem, a corollary (\ref{cor:main}) is that any \emph{topological $G$-manifold} (each fixed set is locally euclidean) with proper action of a Lie group $G$ has the $G$-homotopy type of a countable proper $G$-CW complex.

In particular, we recover E~Elfving's improved thesis (2001): any locally linear $G$-manifold is $G$-homotopy equivalent to a proper $G$-CW complex.
The motivating case of real-analytic $G$-manifolds was studied by S~Illman for decades, who ultimately showed for all Lie $G$ that in fact they admit an essentially unique $G$-triangulation (2000).
In turn, this was an equivariant generalization of the triangulation of differentiable manifolds (Cairns 1934, Whitehead 1940, Whitney 1957).
For topological manifolds, one indeed must work up to homotopy equivalence instead of homeomorphism to a CW complex, since recently there are shown to exist non-triangulable topological manifolds of any dimension $> 3$ (C~Manolescu 2016).

\section{Equivariant absolute neighborhood retracts, I}\label{sec:GANRs1}

\begin{defn}\label{defn:orbit_types}
Let $G$ be topological group.
For a subset $S$ of a $G$-space $X$, write
\[
\cO_G(S) ~:=~ \{ (G_x) ~|~ x \in S \}, \quad\text{where}\quad G_x ~:=~ \{ g \in G ~|~ gx=x \},
\]
for the \textbf{set of $G$-orbit types}, made of the $G$-conjugacy classes of \textbf{isotropy groups}.
For any subgroup $H$ of $G$, the \textbf{$H$-skeleton} ($H$-fixed set) and the \textbf{$H$-stratum} are
\[
S^H ~:=~ \{x \in S ~|~ H \leqslant G_x\} \quad\text{and}\quad S_H ~:=~ \{x \in S ~|~ H = G_x\}.
\]
By a \textbf{$G$-Banach space}, we shall mean a real Banach space $(V,\|\cdot\|)$ equipped with a continuous action $G \x V \longra V$ by isometric linear automorphisms of $(V,\|\cdot\|)$.
\end{defn}

\begin{exm}
Let $p$ be a prime.
Consider the ultrametric $d_p(x,y) := |x-y|_p$ on the compact abelian topological group $\Z_p := \lim_{n\to\infty} \Z/p^n$ induced by the $p$-norm
\[
|z|_p ~:=~ p^{\displaystyle -\sup(\{0\}\cup\{n>0 ~|~ z_n = 0 \in \Z/p^n\})}.
\]
The $\Z_p$-Banach space $C(\Z_p)$ consists of continuous functions $\Z_p \longra \R$ with the sup-norm $\|\cdot\|_\infty$.
There is a compact set in $C(\Z_p)$ with infinitely many orbit types:\footnote{Similarly, for the circle group $U_1$, the sequence $(n \longmapsto f_n)$ with $n$-th term $f_n(z) = \Re(z^n)/n$, which has isotropy group $\Z/n$, converges uniformly in $C(U_1)$ to $f(z) \equiv 0$, which has isotropy $U_1$.}
the closure of the set $\left\{\displaystyle\sum_{n=1}^N \frac{\cos(2\pi z_n / p^n)}{n^2} \right\}_{N>0}$ of partial sums of an infinite series.
It shows the singular set $C(\Z_p)_{sing}$, an $F_\sigma$ subset and linear subspace, is not closed.

Consider Kuratowski's isometric embedding \cite[\S6]{Kuratowski} for the $p$-adic integers:
\[
\Phi: (\Z_p,d_p) \longra (C(\Z_p),\|\cdot\|_\infty) ~;~ x \longmapsto (y \mapsto d_p(x,y)).
\]
For all $x,y \in \Z_p$, note that the midpoint between $\Phi(x)$ and $\Phi(y)$ has trivial isotropy.
Also $\Z_p$ possesses invariant non-ultra metrics averaged from Cantor's discontinuum.
\end{exm}

Recall\footnote{Montgomery and Zippin \cite[1.18]{MZ_book} credit this result to Pontryagin via Weil's book (1940), wherein the latter refers to the former's book (1939).
Later, Pontryagin's second edition credits Kolmogorov for regularity and then unnamed others for complete regularity \cite[\foreignlanguage{russian}{Пример}~32]{Pontryagin2}.} that any topological group is completely regular \cite[\foreignlanguage{russian}{Теорема}~10]{Pontryagin2}.\linebreak
Hence the group is Tikhonov ($T_{3.5}$) if the singleton $\{1\}$ is closed ($T_0$, Kolmogorov).
It has an invariant metric if first-countable $T_0$ by Birkhoff--Kakutani \cite[1.23]{MZ_book}.

\begin{defn}[Palais]\label{defn:Palais}
Let $G$ be a locally compact Hausdorff group.
Let $X$ be a Tikhonov space.
A $G$-action on $X$ is \textbf{Palais} (proper in sense of \cite[1.2.2]{Palais}) if each point $x \in X$ has a neighborhood $U$ such that any $y \in X$ admits a neighborhood $V$ in $X$ satisfying the property that $\langle U,V\rangle_G := \{g \in G ~|~ U \cap gV \neq \varnothing\}$ is precompact.
\end{defn}

Palais' metrization condition \cite[4.3.4]{Palais}  generalizes to a criterion \cite[B]{AdN}.

\begin{lem}[Antonyan--deNeymet]\label{lem:AntonyandeNeymet}
Let $G$ be a locally compact Hausdorff group.
Let $X$ be a metrizable space with Palais $G$-action.
Then $X$ is \textbf{$G$-metrizable} (that is, admits a $G$-invariant metric: $d(gx,gy)=d(x,y)$) if and only if $X/G$ is metrizable.
Moreover, this criterion holds if also $G$ is separable and $X$ is locally separable.
\end{lem}

Recall $X$ is a \textbf{$G$-ANR for the class $\cC$} ($\cC$-absolute neighborhood $G$-retract) if $X$ belongs to $\cC$ and, for any closed $G$-embedding of $X$ into a member of $\cC$, there is a $G$-neighborhood of $X$ with $G$-retraction to $X$.
More generally, $X$ is a \textbf{$G$-ANE for the class $\cC$} ($\cC$-absolute neighborhood $G$-extensor) if, for any member $B$ of $\cC$ and closed $G$-subset $A$ of $B$ and any $G$-map $A \longra X$, there is a $G$-extension $U \longra X$ from a $G$-neighborhood $U$ of $A$ in $B$.
A $G$-ANE need not belong to $\cC$.
Write $\cM$ for the class of metrizable spaces and $G$\text{-}$\cM$ for class of $G$-metrizable Palais $G$-spaces.

Dugundji extension \cite[4.2]{Dugundji} equivariantly generalizes to \cite[4.2]{Abels}  \cite{Antonyan_retracts}.
Abels' proof works for semireflexive $K$-normed linear spaces \cite[Cor~VI:2.7]{Bourbaki_integration}.

\begin{lem}[Abels--Antonyan]\label{lem:extensor}
Let $K$ be a compact Hausdorff group.
Any closed convex invariant subset of a $K$-Banach space is an absolute $K$-extensor for $K\text{-}\cM$.
\end{lem}

Next is an induction \cite[Thm~5]{Antonyan3} and a neighborhood version of \cite[4.4]{Abels}.

\begin{lem}[Abels--Antonyan]\label{lem:induction}
Let $G$ be a locally compact Hausdorff group.
A Palais $G$-space $X$ is a $G\text{-ANE}(G\text{-}\cM)$ if $X \in K\text{-ANE}(K\text{-}\cM)$ for all compact $K<G$.
\end{lem}

\begin{rem}[Abels]\label{rem:Abels}
This holds if true on maximal compact $K<G$, if they exist.
\end{rem}

\begin{prop}\label{prop:hull}
Let $G$ be a locally compact Hausdorff group.
Let $X$ be a Palais $G$-subset of a $G$-Banach space $V$.
Suppose that $X$ is a $G$-neighborhood retract of its closed convex hull $\cco(X)$ in $V$.
Then $X$ is a $G\text{-ANR}$\footnote{A sequence of similar, but non-equivariant, criteria are exercise-statements in \cite[III:\S6]{Hu}.} for the class $G\text{-}\cM$.
\end{prop}

In the proof, note that any closed convex Palais $G$-subset of $V$ is a $G\text{-AE}(G\text{-}\cM)$.
A~Feragen found $\R^n \in G\text{-AE}(X)$ for Palais $X \in T_5$ with $X/G$ paracompact \cite{Feragen}. 

\begin{proof}
There exist a $G$-neighborhood $U'$ of $X$ in $\cco(X)$ and a $G$-retraction $r: U' \longra X$.
Let $K$ be a compact subgroup of $G$.
Suppose $A$ is a closed $K$-subset of a $K$-metrizable $K$-space $B$ and $f: A \longra X \subset \cco(X)$ is a $K$-map.
By Lemma~\ref{lem:extensor}, there exists a $K$-extension $g: B \longra \cco(X)$ of $f$.
Note $U := g\inv(U')$ is a $K$-neighborhood of $A$ in $B$.
Then $r\circ g|_U: U \longra X$ is a $K$-extension of $f$.
Thus $X \in K\text{-ANE}(K\text{-}\cM)$.
Therefore, by Lemma~\ref{lem:induction}, and since the $G$-norm on $V$ induces a $G$-metric on the Palais $G$-space $X$, we obtain $X \in G\text{-ANE}(G\text{-}\cM) \cap (G\text{-}\cM) \subseteq G\text{-ANR}(G\text{-}\cM)$.
\end{proof}

Suitable for our purposes, although there are subsequent variations, the inheritance to fixed sets of being an absolute neighborhood retract is shown in \cite{Smirnov}.

\begin{lem}[Smirnov]\label{lem:Palais}
Let $G$ be a metrizable group.
Suppose $X$ is a $G$-ANR for the class $G\text{-}\cM$.
For any closed subgroup $K$ of $G$, the $K$-skeleton $X^K$ is in $\text{ANR}(\cM)$.
\end{lem}

\section{Non-equivariant construction of $G$-maps}

\begin{defn}\label{defn:slice}
Let $G$ be a topological group, and let $X$ be a $G$-space.
Let $H$ be a subgroup of $G$.
A subset $S$ of $X$ is an \textbf{$H$-slice} if $G S$ is open in $X$ and there exists a $G$-map $f: GS \longra G/H$ such that $S = f\inv(H)$, the preimage of the trivial coset. 
\end{defn}

R~Palais' slice theorem \cite[2.3.1]{Palais} has an approximate version without assuming that $G$ has no small subgroups.
H~Abels was the first to consider non-Lie $G$ \cite[3.3]{Abels} but excluded $O$ below, which we need for Lemma~\ref{lem:comparison}.
Independently, H~Biller \cite[3.8]{Biller_cm} \cite[2.5]{Biller_proper} and S~Antonyan \cite[3.6]{Antonyan2}  prescribed such an $O$.
For any $G$, proper in the senses of Cartan, Bourbaki, and Palais are in \cite[1.1]{Biller_proper}.

\begin{lem}[Abels--Biller--Antonyan]\label{lem:approx_slice}
Let $G$ be a locally compact Hausdorff group.
Let $X$ be a Tikhonov space equipped with a Palais-proper $G$-action.
Let $O$ be a given neighborhood of some point $x$ in $X$.
There exist a compact subgroup $H$ of $G$, a normal subgroup $N \leqslant H$ of $G$ with $G/N$ a Lie group, and an $H$-slice $S$ in $X$ such that $x \in S \subset O$.
($S$ becomes a \textbf{slice at $x$} if $H=G_x$.)
\end{lem}

Palais' proof of the Montgomery--Zippin neighboring-subgroups theorem \cite[4.2.1]{Palais} now similarly works, replacing the use of his slice theorem \cite[2.3.1]{Palais} for Lie groups $G$ with the above approximate one, at the expense of a weaker conclusion.

\begin{cor}\label{cor:MontgomeryZippin}
Let $G$ be a locally compact Hausdorff group.
Let $H$ be a compact subgroup of $G$.
Let $U$ of $1$, and $O$ of $H$, be neighborhoods in $G$.
There exist both a subneighborhood $V \subset O$ of $H$ in $G$ and a compact subgroup $H' \subset O$ of $G$, such that any compact subgroup of $G$ contained in $V$ is $U$-conjugate into $H' \geqslant H$.
\qed
\end{cor}

\begin{rem}
Consider the additive group $G = \Z_p$ of $p$-adic integers.
If $U=G$ and $H=0$, then $H' \neq H$, since there are nontrivial subgroups of $\Z_p$ that are arbitrarily close to the trivial subgroup.
However $H'=H$ works for any Lie $G$ \cite[4.2.1]{Palais}.
\end{rem}

The following rigidity lemma is a consequence of the approximate-slice theorem.

\begin{lem}\label{lem:comparison}
Let $G$ be a locally compact Hausdorff group.
Let $X$ be a $G$-set.
Write $q: X \to X/G$ for $x \mapsto Gx$.
Let $\cT_0 \subseteq \cT_1$ be $G$-invariant topologies on $X$ with $(X,\cT_0)$ Hausdorff and $(X,\cT_1)$ both Tikhonov and Palais.
If $q_*(\cT_0) = q_*(\cT_1)$ then~$\cT_0 = \cT_1$.
\end{lem}

Recall $(X,\cT)$ induces the quotient topology $q_*(\cT) := \{ V \subseteq X/G ~|~ q\inv(V) \in \cT \}$.
A continuous function is \textbf{Bourbaki-proper}\footnote{If the source of the function is Hausdorff and the target is locally compact, this notion becomes: $(\ast)$ the preimage of any compact subset is compact \cite[Proposition~I:10.7]{Bourbaki}.} means that its product with the identity function on each topological space is a closed function \cite[D\'efinition~I:10.1]{Bourbaki}.
A $G$-action on $X$ is \textbf{Bourbaki} means that the following map is Bourbaki-proper:
\[
G \x X \longra X \x X ~;~ (g,x) \longmapsto (x,gx).
\]

\begin{proof}
First assume $G$ is compact\footnote{This part of the argument is a reformulation of the latter half of tomDieck's \cite[Proof~I:7.3]{tomDieck_TG}.} Hausdorff.
Then, since $(X,\cT_1)$ is Hausdorff, the orbit map $q: (X,\cT_1) \longra (X/G,q_*\cT_1)$ is Bourbaki-proper \cite[I:3.6ii]{tomDieck_TG}\footnote{The proof of Part(ii) implicitly assumes locally compact in its characterization $(\ast)$ of `proper'. A faster proof is by applying Part(i) to a product with any space equipped with a trivial $G$-action.}.
Since $q_*\cT_0=q_*\cT_1$, this orbit map is the composite of the continuous functions $\id_X: (X,\cT_1) \longra (X,\cT_0)$ and $q: (X,\cT_0) \longra (X/G,q_*\cT_0)$.
Then, since $(X,\cT_0)$ is Hausdorff, $\id_X: (X,\cT_1) \longra (X,\cT_0)$ is Bourbaki-proper \cite[Proposition~I:10.5d]{Bourbaki}.
Therefore it is a homeomorphism \cite[Proposition~I:10.2]{Bourbaki}.
That is, $\cT_0=\cT_1$.

Now assume $G$ is locally compact Hausdorff.
Fix $O \in \cT_1$.
Let $x \in O$.
Since $(X,\cT_1)$ is Tikhonov with Palais $G$-action, by the approximate-slice theorem (\ref{lem:approx_slice}), there exist both a compact subgroup $H$ of $G$ and an $H$-slice $S$ such that $x \in S \subset O$.
First note $\cT_1 \ni GS = q\inv(GS/G) \in \cT_0$ since $GS/G \in q_* \cT_1 \subseteq q_* \cT_0$.
Write $i: S \longra X$ and $j: GS \longra X$ for the inclusions.
Second, since $q_*\cT_0 = q_*\cT_1$, note
\[
(S,i^*\cT_0)/H ~=~ (G S,j^*\cT_0)/G ~=~ (G S,j^*\cT_1)/G ~=~ (S,i^*\cT_1)/H.
\]
Then, by the previous paragraph applied to the $H$-set $S$, we obtain $i^*\cT_0=i^*\cT_1$.
Since $\cT_0$ and $\cT_1$ are Hausdorff, $G$ is locally compact, and $H$ is closed in $G$, by \cite[Lemma~3.5]{Abels} which has a proof without any potentially implicit use of Tikhonov \cite[3.2a]{Biller_cm}, the $G$-tube $(G S,j^*\cT_k)$ is canonically homeomorphic to the balanced product $G \x_H (S,i^*\cT_k)$ for each $k=0,1$.
Then $j^*\cT_0 = j^*\cT_1$.
So $x \in O \cap GS \in j^*\cT_0$.
Since $GS \in \cT_0$, we further have $O \cap GS \in \cT_0$.
Thus $O \in \cT_0$.
Therefore $\cT_0 = \cT_1$.
\end{proof}

The statement of the following correspondence, for $G$ a compact Lie group, is originally due to Segal--Kosniowski \cite[p90]{Kosniowski}.
Our notation follows tomDieck's proof \cite[I:7.2, I:7.3]{tomDieck_TG} for $G$ a compact Hausdorff group, which we now generalize:
\begin{eqnarray*}
I_G(X,Y) &:=& \{ (x,y) \in X \x Y ~|~ G_x \leqslant G_y \}\\
M_G(X,Y) &:=& I_G(X,Y)/G;
\end{eqnarray*}
$G$ acts on $I_G(X,Y)$ by $g(x,y) := (gx,gy)$ and $\Map(X,Y)$ by $g(f)(x) := g\inv f(gx)$.

\begin{prop}\label{prop:tomDieck}
Let $G$ be a locally compact Hausdorff group.
Let $X$ and $Y$ be Tikhonov spaces with Palais $G$-actions.
Write $\Map(X,Y)^G$ for the set of continuous $G$-functions $X \to Y$, and $\Sec(\pi)$ for the set of continuous sections of the map
\[
\pi: M_G(X,Y) \longra X/G ~;~ G(x,y) \longmapsto Gx.
\]
Then the following well-defined function is a bijection:
\[
\Gamma: \Map(X,Y)^G \longra \Sec(\pi) ~;~ f \longmapsto (Gx \mapsto G(x,fx)).
\]
\end{prop}

\begin{proof}
Consider the pullback $G$-space $Z := \lim(M_G(X,Y) \xrightarrow{~\pi~} X/G \xleftarrow{~q~} X)$.
By \cite[Proposition~I:7.2]{tomDieck_TG}, it suffices to show that this $G$-map is a $G$-homeomorphism:
\[
\phi: I_G(X,Y) \longra Z ~;~ (x,y) \longmapsto (G(x,y), x).
\]
Clearly it is surjective.
If $\phi(x,y)=\phi(x',y')$ then $x=x'$ and $G(x,y)=G(x,y')$, that is, there exists $g \in G$ such that $gx=x$ and $gy=y'$, so $g \in G_x \leqslant G_y$ hence $y=gy=y'$.
Thus $\phi$ is injective.
Therefore the continuous $G$-map $\phi$ is a bijection.

Since the $G$-actions on $X$ and $Y$ are Palais, the diagonal action on $X \x Y$ is Palais \cite[1.3.3]{Palais}.
So the $G$-subspace $I_G(X,Y)$ of $X \x Y$ is Palais \cite[1.3.1]{Palais}.
Therefore its quotient $M_G(X,Y)$ is Tikhonov \cite[1.2.8]{Palais} hence Hausdorff.
Since $X$ is Hausdorff, the subspace $Z$ of $M_G(X,Y) \x X$ is Hausdorff.
Also, since $X$ and $Y$ are Tikhonov, so is $X \x Y$ and hence so is the subspace $I_G(X,Y)$ \cite[33.2]{Munkres}.

Observe the projection $Z/G \longra M_G(X,Y)$ is a homeomorphism \cite[3.25(14)]{tomDieck_TG}.
Furthermore, the induced map $\phi/G: M_G(X,Y) \longra Z/G$ is the continuous inverse.
Write $\cT_1$ for the topology on $I_G(X,Y)$.
Write $\cT_0'$ for the topology on $Z$.
Write $\cT_0 := \phi^*(\cT_0')$ for the $\phi$-induced topology on $I_G(X,Y)$, which is coarser than $\cT_1$.
Therefore, by Lemma~\ref{lem:comparison}, the bijective $G$-map $\phi$ is a homeomorphism.
\end{proof}

Historically, Proposition~\ref{prop:tomDieck} is a useful conversion trick in transformation groups.
Lashof \cite{Lashof} used this trick so Jaworowski could improve from \cite{Jaworowski1} to \cite{Jaworowski2}.
The case of $X$ and $Y$ being the total spaces of principal $G$-bundles is \cite[II:2.6]{Bredon_TG}.

Consequently, we observe an improvement of James--Segal's nonequivariant criterion for $G$-ANEs beyond compact groups $G$ \cite[5.1]{JS2}.
It is a technical device aiding in Theorem~\ref{thm:JamesSegal}.
Denote by $\cC \downarrow B$ the \textbf{overcategory} whose objects are morphisms $\cC \ni C \longra B$ and morphisms $C' \longra C$ form commutative triangles.
We suggest that interested readers acquaint themselves with the further terminology of overspaces in \cite{JS2}, allowing for a parameterized version of ANE theory.

Below, $\cP$ denotes all paracompact Hausdorff spaces; $G\text{-}\cP$ denotes the members of $\cP$ equipped with a Palais $G$-action whose orbit space is also a member of $\cP$.
Also consider the subclass $\cP^*$ of hereditarily $\cP$ spaces (that is, each subspace is a member of $\cP$) and correspondingly $G\text{-}\cP^*$.
Recall that the class $T_{3.5}$ of Tikhonov spaces is hereditary and preserved under taking orbit spaces of Palais actions \cite[1.2.8]{Palais}.

\begin{lem}[James--Segal--Khan]\label{lem:JamesSegal1}
Let $G$ be a locally compact group.
Let $p: E \longra B$ be a $G$-map of $T_{3.5}$ spaces with Palais $G$-actions.
Let $\cC$ be one of $\cP^* \subset \cP \subset T_{3.5}$.
Then $E$ is an absolute (neighborhood) $G$-extensor over $B$ for the class $G\text{-}\cC \downarrow B$ if and only if: for all members $Z \in G\text{-}\cC$, the orbit space $M_G(Z,E)$ is an absolute (neighborhood) extensor over the orbit space $M_G(Z,B)$ for the class $\cC \downarrow M_G(Z,B)$.
\end{lem}

\begin{proof}
The actions being Palais is automatically satisfied if $G$ is compact.
Our variant of \cite[Proposition~5.1]{JS2} further assumes $E,B \in T_{3.5}$ and not only $T_2$.
Modification of the proof of James--Segal ($\cC = \cP$) only requires the replacement of each direction's use of the bijection \cite[I:7.2, I:7.3]{tomDieck_TG} with our Proposition~\ref{prop:tomDieck}.
\end{proof}

\begin{rem}
Unfortunately a gap exists in \cite[Proof~5.2]{JS2} whereby the open sets covering ``$Z$'' (Case~I), locally closed subsets ``$Z_j - Z_{j-1}$'' (Case~II), and open sets ``$U$'' (Case~III) must belong to $\cP$ to apply \cite[3.1]{JS2}.
Hence we prepared the $G\text{-}\cP^*$ version.
This bootstrapping method \cite{Bredon_TG, Jaworowski2} recurs in Proposition~\ref{prop:FTC}.
We write ``extensor'' instead of ``retract,'' as the subquotient $M_G(Z,B)$ may not be a member of $\cP^*$ if $Z, B \in G\text{-}\cP^*$.
For example, the product of the Michael line $\M \in \cP^*$ and the irrationals $\R-\Q \in \cM$ is not a member of $T_4$, so not paracompact \cite{Michael2}.
This also leads to a technical misapplication of \cite[3.1]{JS2} in \cite[Proof~5.2]{JS2}, which we now correct with a minor generalization (compare with \cite[4.4]{Khan_classify}).
\end{rem}

\begin{lem}[James--Segal]\label{lem:JamesSegal2}
Let $p: E \longra B$ and $q: B \longra B_0$ be maps of topological spaces with $B_0 \in \cP$.
For each member $U$ of an open cover of $B_0$, suppose that $E|q\inv(U)$ is an absolute (neighborhood) extensor over $q\inv(U)$ for the class $\cP \downarrow q\inv(U)$.
Then $E$ is an absolute (neighborhood) extensor over $B$ for $\cP \downarrow B$.
\end{lem}

\begin{proof}
The \cite[Proof~3.1]{JS2} for $q=\id_B$ works just as well in this setting.
The reference for Milnor's trick in the class $\cP$ for a \emph{countable} locally finite cover is \cite[p25--26]{Milnor_notes}; Dieudonn\'e's shrinking lemma in the class $T_4$ is \cite[Th\'eor\`eme~6]{Dieudonne}.
\end{proof}

We state another gluing lemma but drop a hypothesis unsatisfied in application.

\begin{lem}[James--Segal]\label{lem:JamesSegal3}
Let $p: E \longra B$ be a map of topological spaces.
Suppose that $E$ is an absolute neighborhood extensor over $B$ for the class $\cP \downarrow B$.
Let $C \subset B$ be closed.
If $E|C$ is an absolute extensor over $C$ and $E|B-C$ is so over $B-C$, then $E$ is moreover an absolute extensor over $B$ for the class $\cP \downarrow B$.
\end{lem}

\begin{proof}
The \cite[Proof~3.2]{JS2} works just as well without the assumption $B \in \cP$.
The only point-set assumption used is that the member ``$Y$'' of $\cP$ belongs to $T_4$.
\end{proof}

We arrive at a corrigendum of \cite[Proposition~4.1]{JS2}, which is James--Segal's $G$-ANE to $G$-AE criterion.
We improve it beyond compact Lie $G$, assuming a bit more point-set topology on $E$ (Tikhonov) and less on $B$ (need not be paracompact).

\begin{thm}[James--Segal--Khan]\label{thm:JamesSegal}
Let $G$ be any Lie group.
Let $p: E \longra B$ be a $G$-map of $T_{3.5}$ spaces with Palais $G$-actions.
Suppose $E$ is an absolute neighborhood $G$-extensor over $B$ for the class $G\text{-}\cP^* \downarrow B$.
For each compact subgroup $H$ of $G$, assume the $H$-skeleton $E^H$ is an absolute extensor over the $H$-skeleton $B^H$ for the class $\cP \downarrow B^H$.
Then $E$ is moreover an absolute $G$-extensor over $B$ for $G\text{-}\cP^* \downarrow B$.
\end{thm}

\begin{proof}
In \cite[4.1]{JS2}, by Lemma~\ref{lem:JamesSegal1} for $\cC = \cP^*$, James--Segal would reduce to a corrected \cite[Proposition~5.2]{JS2}: $M_G(Z,E) \in \text{AE}\left(\cP^* \downarrow M_G(Z,B)\right)$ if $Z \in G\text{-}\cP^*$.

Their reductive proof of this nonequivariant assertion works, with the following adjustments.
In Case~I, as $\cP \subset T_{3.5}$, instead use Palais' slice theorem \cite[2.3.1]{Palais}, and then Lemma~\ref{lem:JamesSegal2} for \cite[3.1]{JS2} with $q$ the projection $M_G(Z,B) \longra Z/G \in \cP$.
In Case~II, apply Lemma~\ref{lem:JamesSegal3} for \cite[3.2]{JS2}, as $(Z_j-Z_{j-1})/G \in \cP$ and $M_G(Z_j,E) \in \text{ANE}\left(\cP\downarrow M_G(Z_j,B)\right)$ by Lemma~\ref{lem:JamesSegal1}.
In Case~III, apply Lemma~\ref{lem:JamesSegal2} for \cite[3.1]{JS2}, as each open subset of $Z/G \in \cP^*$ belongs to $\cP$.
Finally, in their clever Case~IV, as in Case~I use Palais' slice theorem \cite[2.3.1]{Palais}, and that the action is Cartan so each isotropy group ``$G_z$'' is a compact Lie group. In the descending recursion, the appeal to the nonexisting \cite[2.5]{JS2} should instead be to Lemma~\ref{lem:JamesSegal3}.
\end{proof}

\begin{cor}[James--Segal--Khan]\label{cor:JamesSegal}
Let $G$ be any Lie group.
Let $f: A \longra B$ be a $G$-map between members of $G\text{-}\cP^*$ that are absolute neighborhood $G$-extensors for $G\text{-}\cP^*$.
For each compact subgroup $H$ of $G$, assume the induced map $f^H: A^H \longra B^H$ of $H$-skeleta is a homotopy equivalence.
Then $f$ is a $G$-homotopy equivalence.
\end{cor}

\begin{proof}
The same as \cite[Proof~4.2]{JS2}, replacing \cite[4.1]{JS2} with Theorem~\ref{thm:JamesSegal}.
The assumption $A,B \in G\text{-}\cP^*$ is used for $G$-sections to mapping-path spaces ``$W$.''
\end{proof}

Examples in $\cP^*$ are the Michael line $\M$, metrizable spaces, and CW complexes.

\section{Local finiteness of orbit types, I: transtoral induction}

The arguments of G\,D~Mostow inspired this induction scheme \cite[VII:2.1]{Borel_seminar}.

\begin{lem}[Bredon]\label{lem:Bredon}
Let $K$ be a compact Lie group (maybe disconnected).
Let $\cS_0$ be a conjugation-invariant set of closed subgroups of $K$.
Let $T$ be a maximal torus.
If $\{S \cap T ~|~ S \in \cS_0\}$ is finite, then $\cS_0$ has only finitely many conjugacy classes.
\end{lem}

\begin{rem}\label{rem:MaximalTorus}
Any compact Lie group $K$ has a maximal torus unique up to conjugacy \cite[IX:2.2]{Bourbaki2}.
Let $G$ be a Lie group.
The Cartan--Malcev--Iwasawa theorem \cite[6]{Iwasawa} states that there is a conjugacy-unique maximal compact subgroup $K$ of the identity component $G^e$ of $G$.
Then $G$ has a conjugacy-unique maximal torus~$T$.
\end{rem}

\begin{defn}\label{defn:compactlysupported}
Let $\cS$ be a collection of subsets of a topological group $G$.
Herein, we say that $\cS$ is \textbf{compactly supported} if it is closed under conjugation in $G$ and there exists a compact set $C$ in $G$ such that each subset $S \in \cS$ is conjugate into $C$.
\end{defn}

\begin{rem}\label{rem:Mostow}
Let $G$ be a \emph{virtually connected} Lie group, that is, with only finitely many components.
Let $\cS$ be a conjugation-invariant set of compact subgroups of $G$.
Mostow extends the theorem of Remark~\ref{rem:MaximalTorus}: there is a maximal compact subgroup $K$ of such $G$ unique up to conjugacy \cite[3.1]{Mostow_selfadjoint}.
Then $\cS$ is compactly supported; this fails for the group $G = \bigoplus_{n \in \N} \Z/n$, which has no maximal compact subgroup.
\end{rem}

Our notion of $(S)$-maximal below is stronger than maximal in $(\cpt_S(G),\leqslant)$.

\begin{defn}\label{defn:partialorder}
Let $G$ be an arbitrary Lie group.
For any compact subgroups $H$ and $K$ of $G$, write $H \preccurlyeq K$ to mean that a $G$-conjugate of $H$ is a subgroup~of~$K$.
Since closed subgroups of $G$ are Lie (Cartan's theorem \cite[27]{CartanE}) and compact Lie groups are cohopfian (injective endomorphisms are bijective), $\preccurlyeq$ satisfies antisymmetry.
Hence $\preccurlyeq$ is a \emph{partial order} on the set $\cpt(G)$ of compact subgroups~of~$G$.

Let $S$ be a subset of $G$.
Write $\cpt_S(G)$ for the subset of $\cpt(G)$ consisting of compact subgroups of $G$ contained in $S$; it may be empty.
By $H \in \cpt_S(G)$ is \textbf{$(S)$-maximal} we mean: if $H \preccurlyeq K \in \cpt_S(G)$ then $K$ is a $G$-conjugate of $H$.
\end{defn}

For Lie groups with any discrete $\pi_0$, here is a downgrade of Mostow's theorem.

\begin{prop}\label{prop:Mostow}
Let $G$ be an arbitrary Lie group.
Let $C$ be a compact subset of $G$.
Any compact subgroup of $G$ in $C$ conjugates into a $(C)$-maximal compact subgroup.
Moreover, the set of conjugacy classes of $(C)$-maximal compact subgroups is finite\footnote{The shape of the compact subset effects uniqueness, for example $G = U_1 \x U_1$ and $C=U_1 \vee U_1$.}.
\end{prop}

\begin{proof}
Recall that the set $\Cpt(X)$ of nonempty compact subsets of a metric space $(X,d)$ is topologized by the Pompeiu--Hausdorff metric \cite[\S 21]{Pompeiu} \cite[p293]{Hausdorff1}:
\[
d_{PH}(A,B) ~:=~ \max\left\{ \sup_{a \in A} d(a,B), \sup_{b \in B} d(b,A) \right\} \quad\text{where}\quad d(x,S) ~:=~ \inf_{y \in S} d(x,y).
\]
Since $C$ is compact, $\Cpt(C)$ is compact \cite[Satz~28:VI]{Hausdorff2}.
Since multiplication and inversion in $G$ are continuous, it follows that any Cauchy sequence in $\cpt_C(G)$ converges in $\cpt_C(G)$, as a limit already exists in $(\Cpt(C),d_{PH})$.
Then $\cpt_C(G)$ is closed in $\Cpt(C)$ so compact\footnote{This step is a Lie-group analogue of Blaschke's selection theorem for convex sets \cite[18.I]{Blaschke}.} with respect to the metric $d_{PH}$.

Fix $K_0 \in \cpt_C(G)$.
Let $\{K_\alpha\}_{\alpha\in\cL}$ be a nonempty chain with each $K_\alpha \succcurlyeq K_0$ in the partially ordered set $(\cpt_C(G), \preccurlyeq)$ of Definition~\ref{defn:partialorder}.
(Here, a \emph{chain} is a subset $\cL$ of $\cpt_C(G)$ whose partial order restricts to a linear order on $\cL$.)
Since $\cpt_C(G)$ is sequentially compact and every chain is a (Moore--Smith) net, by a theorem of Kelley \cite[24]{Kelley_nets}, the chain $\{K_\alpha\}_{\alpha \in \cL}$ has a \emph{convergent subnet}.
That is, there exist $K \in \cpt_C(G)$ with $K \succcurlyeq K_0$ and cofinal $\cL' \subset \cL$ satisfying: for each $\eps>0$, there exists $\beta \in \cL'$ such that $d_{PH}(K_\alpha,K) < \eps$ for all $\beta \preccurlyeq \alpha \in \cL'$.
Since $K$ is compact and $G$ is Lie, by Montgomery--Zippin's neighboring-subgroups theorem \cite{MZ_neighbor}, there exists $\eps>0$ satisfying: if $K' \in \cpt(G)$ and $d_{PH}(K',K) < \eps$ then $K'$ is conjugate into $K$.
Then there exists a confinal subchain $\cL' \subset \cL$ and $\beta \in \cL'$ such that $K_\alpha \preccurlyeq K$ for all $\beta \preccurlyeq \alpha \in \cL'$.
Since $\cL'$ is cofinal in the linearly ordered $\cL$, it follows that $K_\alpha \preccurlyeq K$ for all $\alpha \in \cL$.
Therefore, by Zorn's lemma \cite{Zorn}, any element $K_0$ of $\cpt_C(G)$ conjugates into a $(C)$-maximal element.

Assume the set $\cM$ of $(C)$-maximal compact subgroups of $G$ has infinitely many conjugacy classes.
By the axiom of choice, there is an infinite subset $\cM_0$ of $\cM$ with pairwise distinct conjugacy classes.
Since $\cpt_C(G)$ is compact metric, $\cM_0$ has an accumulation point $K$.
Again, by the neighboring-subgroups theorem \cite{MZ_neighbor}, there exists $\eps>0$ satisfying: if $K' \in \cpt(G)$ and $d_{PH}(K',K)<\eps$ then $K' \preccurlyeq K$.
There are distinct $K', K'' \in \cM_0$ that are $\eps$-close to $K$.
Then $K' \preccurlyeq K$ and $K'' \preccurlyeq K$.
By $(C)$-maximality, note $(K')=(K)=(K'')$, which contradicts the distinctness of conjugacy classes in $\cM_0$.
Therefore $\cM$ has only finitely many conjugacy classes.
\end{proof}

We generalize the induction scheme of Lemma~\ref{lem:Bredon}, to be useful for our purposes.

\begin{cor}\label{cor:transtoral}
Let $G$ be a Lie group (maybe disconnected).
Let $\cS$ be a compactly supported set of closed subgroups of $G$.
Let $T$ be a maximal torus in $G$ (see \ref{rem:MaximalTorus}).
If $\{S \cap T ~|~ S \in \cS\}$ is finite, then $\cS$ consists of only finitely many conjugacy classes.
\end{cor}

\begin{proof}
The conclusion is immediate from Proposition~\ref{prop:Mostow} and Lemma~\ref{lem:Bredon}.
\end{proof}

The following observation is recorded in the literature by R~Palais \cite[1.7.27]{Palais_book}.
His proof relies on the Peter--Weyl theorem (1927) and C-T~Yang's theorem \cite{Yang}.

\begin{thm}[Palais]\label{thm:Palais}
Any compact Lie group $K$ has only countably many conjugacy classes of closed subgroups.
(Recall that $K$ has only finitely many components.)
\end{thm}

We provide an alternative proof that flows from the above-used first principles, avoiding the aforementioned results in harmonic analysis and differential geometry.

\begin{proof}
Inductively assume the statement is true for compact Lie groups of either lesser dimension or fewer components than $K$, so in particular for all proper closed subgroups of $K$ \cite[27]{CartanE}; the basic case is the trivial group.
For each integer $n>0$, define $X_n := \cpt(K) - B(K,1/n)$, the complement of the open ball of radius $1/n$ centered at the point $K$ in the compact metric space $\cpt(K)$ from Proof~\ref{prop:Mostow}.

By the neighboring-subgroups theorem\footnote{Similarly inspired by Mostow, Palais' later proof \cite[4.2]{Palais} of it avoids differential geometry.} \cite{MZ_neighbor}, for each compact subgroup $H$ of $K$, there exists $\eps_H>0$ such that $d_{PH}(H',H)<\eps_H$ implies $H' \preccurlyeq H$.
Since $X_n$ is compact and consists of elements satisfying the inductive hypothesis, we obtain that $X_n$ has only finitely many conjugacy classes of elements.
Therefore $\cpt(K) = \{K\} \cup \bigcup_{n=1}^\infty X_n$ has only countably many conjugacy classes of elements.
\end{proof}

Lastly, we now recover the Lie case of \cite[3.1]{AAV}, which more generally states that any locally compact Hausdorff group $G$ has at most its \textbf{weight} $wG$ (that is, the minimum cardinality of a base for the topology) for the number of conjugacy classes of compact subgroups $H$ having coset space $G/H$ being finite-dimensional.

\begin{cor}[Antonyan--Antonyan--Varela-Velasco]\label{cor:ccc}
Any Lie group $G$ has only countably many conjugacy classes of compact subgroups.
\end{cor}

\begin{proof}
The Lindel\"of and locally compact space $G$ is \textbf{hemicompact} \cite[8:a]{Arens_hemicompact}.
In other words, $G = \bigcup_{n=1}^\infty C_n$ is the ascending union of compact sets (so $\sigma$-compact) such that any compact set in $G$ is contained in some member $C_n$.
So
\[
\cpt(G) ~=~ \bigcup_{n=1}^\infty \cpt_{C_n}(G).
\]
By Proposition~\ref{prop:Mostow} then Theorem~\ref{thm:Palais}, each $\cpt_{C_n}(G)$ has only countably many conjugacy classes of elements.
Therefore $\cpt(G)$ does also.
\end{proof}

\section{Local finiteness of orbit types, II: miscellaneous applications}

Any abelian Lie group is isomorphic to the finite product of cyclic groups (finite or not), the multiplicative group $U_1$ of unit-norm complex numbers, and the additive group $\R$ of real numbers (for example, \cite[0:5.4]{Bredon_TG}).
That is, it is the product of a finitely generated abelian group, a finite-dimensional torus, and a euclidean space.

\begin{lem}\label{lem:Banach_simplex}
Let $A$ be an abelian Lie group.
Let $(V,\|\cdot\|)$ be an $A$-Banach space.
Any convex simplex in a Palais $A$-subset $V_0$ of $V$ has only finitely many orbit types.
\end{lem}

The convex simplex is the convex hull on only finitely many extremal points.

\begin{proof}
Let $\Delta$ be a convex simplex in $V_0$.
We shall proceed by induction on $\dim(\Delta)$.

Suppose $\dim(\Delta)=1$.
Assume the set $\cO_A(\Delta)$ of $A$-orbit types (\ref{defn:orbit_types}) is infinite.
Since $\Delta$ is compact, there exist $x \in \Delta$ and a sequence $\{x_n\}$ in $\Delta$ with $\|x_n-x\| \searrow 0$ and the isotropy groups $A_{x_n}$ being distinct (that is, non-conjugate in abelian $A$).
Since the action of $A$ on $V_0$ is Palais, there exists an $A_x$-tube at $x$ (\ref{defn:slice}, \ref{lem:approx_slice}), that is an $A$-neighborhood $U$ of $Ax$ in $V_0$ and $A$-map $f: U \longra A/A_x$ \cite[2.1.1]{Palais}.
There is $N$ so that, for all $n \geqslant N$, we have $x_n \in U$ hence $A_{x_n} \leqslant A_x$ as $A$ is abelian.
The isotropy of any point on the line from such $x_n$ to $x$ contains $A_{x_n} \cap A_x = A_{x_n}$.
Then $A_{x_N} \leqslant A_{x_{N+1}} \leqslant A_{x_N}$, by collinearity in $\Delta$.
This contradicts $A_{x_N} \neq A_{x_{N+1}}$.

Inductively assume that the lemma is true for all convex simplices in $V_0$ of dimension $d > 0$.
Suppose $\dim(\Delta)=d+1$.
Assume that $\cO_A(\Delta)$ is infinite.
Since $\Delta$ is compact, there exist $x \in \Delta$ and a sequence $\{x_n\}$ in $\Delta$ with $\|x_n-x\| \searrow 0$ and the isotropy groups $A_{x_n}$ being distinct.
Since the action of $A$ on $V_0$ is Palais, there exists an $A_x$-tube $(U,f)$ at $x$ (\ref{defn:slice}, \ref{lem:approx_slice}).
By the pigeonhole principle, there exist a convex $(d+1)$-simplex $\Delta'$ and a $d$-dimensional face $\Delta''$ of $\Delta'$ such that $x \in \Delta' \subset \Delta \cap U$ and, for infinitely many $n$, the line $L_n$ from $x$ to $x_n$ intersects $\Delta''$.
Re-index so that this statement is true for all $n$.
For each $n$, write $\{y_n\} := L_n \cap \Delta''$.
By inductive hypothesis, $\cO_A(\Delta'')$ is finite.
So the subset $\cO_A\{y_n ~|~ n\in \N\}$ is finite.
However, by the collinearity argument of the basic step, within any $A_x$-tube through $x$, the line from $x$ to any other point in $U$ has constant isotropy in $U$ away from $x$.
Then $A_{x_n} = A_{y_n}$ in $L_n$ for each $n$.
So $\cO_A\{x_n ~|~ n \in \N\}$ is finite, a contradiction.
\end{proof}

\begin{cor}\label{cor:Banach_simplex}
Let $G$ be an arbitrary Lie group.
Let $(V,\|\cdot\|)$ be a $G$-Banach space.
Any convex simplex in a Palais $G$-subset consists of only finitely many orbit types.
\end{cor}

\begin{proof}
Let $\Delta$ be a convex simplex in Palais $V_0$.
Since the $G$-space $V_0$ is Bourbaki \cite[1.6c]{Biller_proper}, the set $C := \{g \in G ~|~ g \Delta \cap \Delta \neq \varnothing \}$ is compact \cite[I:3.21]{tomDieck_TG}.
Then $\bigcup \cO_G(\Delta)$ is supported in $C$.
So, by Corollary~\ref{cor:transtoral} and Lemma~\ref{lem:Banach_simplex}, we are done.
\end{proof}

\begin{rem}
If $K$ is compact and $\dim\,V$ is finite, by covering the unit sphere by orthogonal-action charts \cite[4]{Bochner}, inductively any orthogonal $K$-representation $V$ has finitely many orbit types \cite[Exercise II:2]{Bredon_TG}.
Using a locally finite covering by linear $G_x$-tubes, any Palais locally linear $G$-manifold $M$ \emph{locally} has finitely many orbit types; the original case of compact group $K$ and smooth $K$-manifold is due to C-T~Yang~\cite{Yang}.
Yang's result is a step in Lemma~\ref{lem:Bredon}, used for Corollary~\ref{cor:Banach_simplex}.
\end{rem}

Next, we shall explore $\cO_G$ for the \textbf{convex hull} $\co(C)$ of \emph{infinite} subsets $C \subset V$: that is, $\co(C)$ is the intersection of all convex sets of $V$ containing $C$ \cite[3.19a]{Rudin}.
Equivalently, $\co(C)$ is the set of (finite) convex linear combinations of points in $C$.
Recall that there exist compact $C$ whose $\co(C)$ is not closed \cite[Exercise~3:20]{Rudin}.
Write $\cco(C)$ for the \textbf{closed convex hull}, that is, the metric closure in $V$ of $\co(C)$.
Equivalently, $\cco(C)$ is the intersection of all closed convex sets of $V$ containing $C$.
For any subset $C \subset V$ and $\eps>0$, consider the closed neighborhood $\cD_\eps(C)$ in $\cco(C)$:
\[
\cD_\eps(C) ~:=~ \{ y \in \cco(C) ~|~ d(y,C) \leqslant \eps \}.
\] 

\begin{lem}\label{lem:singleorbit}
Let $A$ be an abelian Lie group.
Let $(V,\|\cdot\|)$ be an $A$-Banach space.
Let $V_0$ be a convex Palais $A$-subset of $V$.
Let $C \subset \Int\,V_0$ be a compact set of only a single orbit type.
There exists $\eps>0$ such that $\cD_\eps(C)$ has only a single orbit type.
\end{lem}

\begin{proof}
Since the action of $A$ on $\cco(C) \subset V_0$ is Palais, at each point $x \in C$ there is an $A_x$-tube $(U_x \subset \cco(C),f_x: U_x \longra A/A_x)$ \cite[2.3.1]{Palais}.
Since $C$ is compact, there exists $\eps>0$ such that the open neighborhood $\bigcup_{x \in C} U_x$ of $C$ in $\cco(C)$ contains the closed neighborhood $\cD_\eps(C)$ \cite[Exercise~27:2d]{Munkres}.
We show $\cO_A(\cD_\eps C) = \cO_A(C)$.

Since $A$ is abelian, each point of $C$ has common isotropy $K \leqslant A$.
If $y \in \co(C)$, note $A_y \geqslant K$.
If $y \in \cco(C)-\co(C)$, there exist an $A_y$-tube $(U_y \subset \cco(C),f_y)$ and a point $z \in U_y \cap \co(C)$, so $A_y \geqslant A_z \geqslant K$.
Thus, for each $y \in \cD_\eps(C)$, we have both $A_y \geqslant K$ and a point $x \in C$ such that $y \in U_x$, so $A_y \leqslant A_x = K$, hence $A_y=K$.
\end{proof}

\begin{cor}
Let $G$ be an arbitrary Lie group.
Let $(V,\|\cdot\|)$ be a $G$-Banach space.
Let $V_0$ be a convex Palais $G$-subset of $V$.
Let $T$ be a maximal torus in $G$ (see \ref{rem:MaximalTorus}).
Let $C \subset \Int\,V_0$ be a compact set of only a single $T$-orbit type.
There exists $\eps>0$ such that $\cD_\eps(C) \supset C$ has only finitely many $G$-orbit types.
\end{cor}

\begin{proof}
By Lemma~\ref{lem:singleorbit}, there exists $\eps>0$ so that $\cD_\eps(C)$ has a single $T$-orbit type.
Since $\cco(C)$ is compact \cite[3.20c]{Rudin}, the above closed set is also \cite[26.2]{Munkres}.
So, since $\cD_\eps(C) \subset V_0$ is compact and the $G$-space $V_0$ is Bourbaki \cite[1.6c]{Biller_proper}, the subset $C' := \{g \in G ~|~ g \cD_\eps(C) \cap \cD_\eps(C) \neq \varnothing \}$ of $G$ also is compact \cite[I:3.21]{tomDieck_TG}.
Then $\bigcup \cO_G(\cD_\eps C)$ is supported in $C'$.
So, by Corollary~\ref{cor:transtoral}, we are done.
\end{proof}

\section{Equivariant absolute neighborhood retracts, II}\label{sec:GANRs2}

For inductive methods of proof, this technical notion was introduced in \cite{Jaworowski2}.

\begin{defn}[Jaworowski]\label{defn:Jaworowski}
Any $G$-space $X$ having \textbf{finite structure} means that $X$ has only finitely many orbit types and, for each orbit type $(H)$, the quotient map $X_{(H)} \longra X_{(H)}/G$ is a $G/H$-fiber bundle with only finitely many local trivializations.
Here $(H)$ is the conjugacy class of $H$ in $G$ and $X_{(H)} := \{x \in X ~|~ (G_x)=(H) \}$.
\end{defn}

It led to a characterization \cite[4.2]{Jaworowski2} that was the culmination of three papers.
Much later, the separable hypothesis was able to be removed by \cite[1.6]{AAMV}.

\begin{lem}[Jaworowski--Antonyan \emph{et al.}]\label{lem:Jaworowski}
Let $K$ be a compact Lie group.
Let $X$ be a $K$-metrizable space of finite structure.
Then $X$ is a $K$-ANE for the class $K\text{-}\cM$ if and only if, for each closed subgroup $H$ of $K$, the $H$-skeleton $X^H$ is ANR for $\cM$.
\end{lem}

The following shorter property occurs in the $K$-metric spaces of \cite[4.1]{AAMV}.

\begin{defn}\label{defn:FTC}
A $G$-space is FTC (\textbf{finite trivializing covers}) if it satisfies the definition of finite structure (\ref{defn:Jaworowski}) without assumption of finitely many orbit types.
\end{defn}

\begin{rem}[Antonyan \emph{et al.}]
In a different direction than ours in Theorem~\ref{thm:criterion}, the generalization \cite[4.1]{AAMV} of Lemma~\ref{lem:Jaworowski} only assumes that $X$ is FTC.
\end{rem}

We affirm \cite[Question~6.5]{AAMV} by removal of the linearity of the Lie group.

\begin{thm}\label{thm:criterion}
Let $G$ be an arbitrary Lie group.
Let $T$ be a maximal torus in $G$.
Let $X$ be a Palais $G$-metrizable space supporting only finitely many $T$-orbit types.\footnote{Assume $X$ is FUI (see Definition~\ref{defn:FUI}).  Then $X$ has only finitely many $G$-orbit types (\ref{lem:FUI_cs}, \ref{cor:transtoral}). Moreover, it suffices to check the $K$-FTC hypothesis only on the upper bounds $K_1, \ldots, K_n$~(\ref{rem:Abels}).}
Suppose $X$ is $K$-FTC for all compact subgroups $K < G$.
Then $X$ is a $G$-ANR for the class $G\text{-}\cM$ if and only if: $X^K$ is ANR for the class $\cM$ for any compact $K < G$.
\end{thm}

\begin{proof}
Let $K$ be any compact subgroup of $G$.
Let $T'$ be a maximal torus in $K$.
Since the maximal tori of $G$ are unique up to conjugacy (\ref{rem:MaximalTorus}), we may assume that $T' \leqslant T$.
Since $T'$ is central in $T$, there exists a restriction $\cO_T(X) \longra \cO_{T'}(X)$, which is automatically surjective.
Hence $\cO_{T'}(X)$ is finite (compare \cite[1.7.30]{Palais_book}).
So $\cO_K(X)$ is finite (\ref{lem:Bredon}).
Since $X$ is $K$-FTC, we conclude $X$ has $K$-finite structure.

First, suppose that $X \in G\text{-ANR}(G\text{-}\cM)$.
Then $X^K \in \text{ANR}(\cM)$ by Lemma~\ref{lem:Palais}.

Conversely, suppose that $X^H \in \text{ANR}(\cM)$ for any compact $H<G$.
Since $X$ has $K$-finite structure and $X^H \in \text{ANR}(\cM)$ for all closed $H\leqslant K$, $X \in K\text{-ANE}(K\text{-}\cM)$ by Lemma~\ref{lem:Jaworowski}.
Since this is true for all $K$, $X \in G\text{-ANR}(G\text{-}\cM)$ by Lemma~\ref{lem:induction}.
\end{proof}

Recently, Lemma~\ref{lem:Jaworowski} is generalized to include linear Lie groups \cite[6.1]{AAMV}.

\begin{cor}[Antonyan--Antonyan--Mata-Romero--Vargas-Betancourt]\label{cor:AAMV_criterion}
Let $L$ be a linear Lie group.
Let $X$ be any Palais $L$-metrizable space having finite structure.
Then $X$ is an $L$-ANE for the class of Palais $L$-metrizable spaces if and only if, for each compact subgroup $K$ of $L$, $X^K$ is an ANR for the class of metrizable spaces.
\end{cor}

The proof of this as a corollary shall appear below, after the following key lemma.
Linearity of $L$ forces keeping the hypothesis of $X$ under restriction \cite[6.4]{AAMV}.

\begin{lem}[Antonyan--Antonyan--Mata-Romero--Vargas-Betancourt]\label{lem:AAMV_restriction}
Let $L$ be a linear Lie group.
Let $X$ be a Palais $L$-metrizable space of finite structure.
For any compact subgroup $K$ of $L$, the restriction of the $K$-action on $X$ has finite structure.
\end{lem}

\begin{proof}[Proof of Corollary~\ref{cor:AAMV_criterion}]
Let $K$ be any compact subgroup of $L$.
Since $X$ is assumed to have $L$-finite structure, by Lemma~\ref{lem:AAMV_restriction}, it has both $T$-finite structure and $K$-finite structure.
Therefore $X$ has only finitely many $T$-orbit types and is $K$-FTC.
Consequently, the $L$-ANE criterion follows immediately from Theorem~\ref{thm:criterion}.
\end{proof}

Now, we turn to relating FTC to finite covering dimension.
The next two results are folklore recorded by G~Bredon in the case of $P$ being a finite simplicial complex.
In our generalization, a \textbf{metric polytope} is a simplicial complex with $L^1$-metric topology; it is \textbf{full} means that it contains the simplex spanned by any finite subset of vertices.
Observe that if $P$ is locally finite then it has no infinite full subpolytope; note a counterexample to the converse is the metric cone on the simplicial real line.

\begin{lem}[Bredon]\label{lem:polytope}
Let $A$ a closed subset of a paracompact Hausdorff space $X$.
Let $P$ be a metric polytope\footnote{For locally finite complexes, the CW topology and the euclidean-metric topology are equal.} with only countably many vertices and no infinite full subpolytope.\footnote{The homotopy type of a countable CW complex contains such a polytope $P$ \cite[Thm~13]{Whitehead1}.}
Suppose $X$ has covering dimension $\leqslant n$ and $P$ is $(n-1)$-connected for some $n \in \N$.
Any continuous function $A \longra P$ extends to a continuous $X \longra P$.
\end{lem}

\begin{proof}
Do verbatim with \cite[II:9.1]{Bredon_TG}, using $P \in \text{ANE}(T_4)$ \cite[III:11.7d]{Hu}.
\end{proof}

\begin{cor}[Bredon]\label{cor:polytope}
Let $n \in \N$.
Let $A$ be a closed subset of a paracompact Hausdorff space $B$ with $\dim(B) \leqslant n$.
Let $F$ be an $(n-1)$-connected, countable, metric polytope with no infinite full subpolytope.
Let $F \longra E \longra B$ be a fiber bundle.
Any continuous section $A \longra E$ extends to a continuous section $B \longra E$.
\end{cor}

\begin{proof}
Do verbatim with \cite[II:9.2]{Bredon_TG}, using Lemma~\ref{lem:polytope} for \cite[II:9.1]{Bredon_TG}.
\end{proof}

We arrive at a classification theorem that becomes Bredon's \cite[II:9.3, 9.7i]{Bredon_TG} when $G$ is a compact Lie group and his isotropy groups are limited to be trivial.
In turn, Bredon's theorem generalized Palais' \cite[2.6.2]{Palais_book} from requiring $B$ to be locally compact, second countable, and Hausdorff \cite[2.1.1]{Palais_book} to $B$ paracompact.

\begin{thm}\label{thm:classify}
Let $G$ be any Lie group.
Recall Milnor's join $E_n G := G^{\circ (n+1)}$ with coarse topology \cite{Milnor_univ2}.
Let $G \longra E \xrightarrow{~\pi~} B$ be a principal $G$-bundle with $B$ paracompact Hausdorff and of covering dimension $\leqslant n$.
There are a map $f: B \longra B_n G := E_n G /G$ and a $G$-homeomorphism $E \longra f^*(E_n G)$ over the identity on $B$.
\end{thm}

Relatively further, if $\dim(B) < n$ then isomorphic bundles have homotopic maps.

\begin{proof}
Firstly, the $(n+1)$-fold coarse join $E_n G$ is $(n-1)$-connected \cite[2.3]{Milnor_univ2}.

Since $G$ is a manifold of class $C^1$, it admits a triangulation \cite[Theorem~7]{Whitehead_triangulability}.
Since $G$ is separable and locally compact, the triangulation is countable and locally finite.
Since the coarse (which is metrizable \cite[\S3.1]{Khan_classify}) join of two metric polytopes each with no infinite full subpolytope has this property, the countable metric polytope $E_n G$ has no infinite full subpolytope (which is false for only locally finite).

Then, by Corollary~\ref{cor:polytope}, the bundle $E_n G \longra E_n G \x_G E \longra B$ admits a section.
This section corresponds to a $G$-map $E \longra E_n G$ \cite[II:2.6]{Bredon_TG}, which induces a map $f: B \longra B_n G$, yielding a $G$-map $E \longra f^*(E_n G)$ inducing the identity on $B$.
In order to show this continuous bijection is open, it remains to apply Lemma~\ref{lem:comparison}.

Since $G$ and $B$ are Hausdorff, so is $f^*(E_n G)$ by local triviality.
Since $G$ is locally compact and $B$ is Hausdorff, the $G$-action on $E$ is Palais~(\ref{defn:Palais}) by local triviality.
We now explicitly show that $E$ is Tikhonov.
Let $C \subset E$ be closed and $x \in E-C$.
There are a neighborhood $U$ of $\pi(x)$ in $B$ and a $G$-homeomorphism $\phi: \pi\inv(U) \longra G \x U$.
Since $B$ is regular \cite[1a]{Dieudonne}, there is a subneighborhood $V$ of $x$ in $B$ with $\ol{V} \subset U$.
Then, since $B$ is normal \cite[1b]{Dieudonne}, by Urysohn's lemma \cite[25]{Urysohn}, there exists a map $\beta: B \longra [0,1]$ with $\beta(\ol{V}) = \{1\}$ and $\beta(B-U) = \{0\}$.
Since $G$ and $U$ are Tikhonov, so is the product $G \times U$ \cite[1.5.8, 2.3.11]{Engelking2}.
Then there is a map $\alpha: G \times U \longra [0,1]$ with $\alpha(\phi(x))=1$ and $\alpha(\phi(C \cap \pi\inv U)) = \{0\}$.
So the map $\gamma := (\alpha\circ\phi) \cdot (\beta\circ\pi): E \longra [0,1]$ satisfies $\gamma(x)=1$ and $\gamma(C)=\{0\}$.
Thus $E$ is Tikhonov.
Therefore $E \longra f^*(E_n G)$ is a $G$-homeomorphism, by Lemma~\ref{lem:comparison}.
\end{proof}

\begin{rem}
For $G$ a compact Lie group, an $n$-classifying bundle exists for $B$ a finite polytope \cite[19.6]{Steenrod_book}.
Our theorem was not observed after the construction of $E_n(G)$ in  \cite{Milnor_univ2}, but Dold \cite[7.6]{Dold} implies it classifies over paracompact $B$ with $\dim(B)\leqslant n$ replaced by $B$ locally a neighborhood retract of euclidean $n$-space.
The relaxed case $n=\aleph_0$ \cite[7.5]{Dold} is easier to achieve \cite[II]{tomDieck_numerable} \cite[4.12.2]{Husemoller}.
\end{rem}

For compact Lie $G$, \cite[4]{Lashof} and \cite[1.3]{Jaworowski2} asserted a version of the following.

\begin{prop}\label{prop:FTC}
Let $G$ be an arbitrary Lie group.
Let $X$ be a Palais $G$-space with orbit space $X/G$ both of finite covering dimension and \textbf{hereditarily paracompact} (that is, all subspaces are paracompact).
Then $X$ satisfies the FTC property (\ref{defn:FTC}).\footnote{\emph{A fortiori,} the proof shows each $X_{(H)}$ has a trivializing cover by $1 + \dim(X/G)$ open sets.}
\end{prop}

\begin{proof}
Let $H$ be a compact subgroup of $G$.
By Cartan's closed-subgroup theorem \cite[27]{CartanE}, both $H$ and its normalizer $N_G H$ are Lie.
Consider the Weyl group $W_G H := N_G H / H$.
Note $X_{(H)} = G/H \x_{W_G H} X_H$ as $G$-spaces.
Since the induced action of the Lie group $W_G H$ on $X_H$ is Palais \cite[1.3.1]{Palais}, $X_H$ is covered by $H$-tubes \cite[2.3.1]{Palais}.
Therefore $X_H$ is the total space of a principal $W_G H$-bundle.

Since implicitly $X$ is Tikhonov (\ref{defn:Palais}), so is $X/G$ \cite[1.2.8]{Palais}.
Since $X/G$ is hereditarily paracompact hence totally normal, by Dowker's monotonicity of dimension \cite[2.8]{Dowker4}, $\dim(X_H/W_G H) \leqslant \dim(X/G)$.
Also, Milnor covers the principal bundle $E_n(W_G H)$ with $n+1$ explicit local trivializations \cite[Proof~3.1]{Milnor_univ2}.
Then, since the subspace $X_H/W_G H$ of $X/G$ is paracompact Hausdorff, Theorem~\ref{thm:classify} implies that $X_H$ has a finite cover by local trivializations.
Therefore $X$ is FTC.
\end{proof}

\begin{rem}[Stone]\label{rem:Stone}
Any metrizable space is hereditarily paracompact \cite{Stone1}.
\end{rem}

Here is a more practical form of Theorem~\ref{thm:criterion} that is devoid of mention of FTC.
For $G$ compact, Jaworowski's $G$-ENR theorem \cite{Jaworowski1} has $X$'s hypotheses below.

\begin{thm}\label{thm:Jaworowski}
Let $G$ be an arbitrary Lie group.
Let $X$ be a Palais $G$-metrizable space.
Suppose $X$ is separable, locally compact, and of finite covering dimension.
Let $T$ be a maximal torus in $G$.
Assume $\cO_T(C)$ is finite for each compact $C \subset X$.
Then $X$ is $G$-ANR for the class $G\text{-}\cM$ of Palais $G$-metrizable spaces if and only if, for each compact subgroup $K$ of $G$, the $K$-skeleton $X^K$ is an ANR for the class $\cM$.
\end{thm}

\begin{proof}
Let $K$ be a compact subgroup of $G$.
In particular, $K$ is Lie \cite[27]{CartanE}.

First, suppose that $X \in G\text{-ANR}(G\text{-}\cM)$.
Then $X^K \in \text{ANR}(\cM)$ by Lemma~\ref{lem:Palais}.

Conversely, suppose that $X^H \in \text{ANR}(\cM)$ for each compact $H<G$.
Let $x \in X$.
Since $X$ is locally compact, there is a precompact open neighborhood $U_x$ of $x$ in $X$.
Since $\cO_T(\ol{U_x})$ is assumed to be finite, as in Proof~\ref{thm:criterion}, we have that $\cO_K(\ol{U_x})$ is finite.
Then $\cO_K(K U_x)$ is finite.
Since $X$ is separable metric, Palais shows $\dim(X/K) \leqslant \dim(X)$ \cite[1.7.32]{Palais_book}.
Since $X/K$ is metrizable hence perfectly normal (that is, $T_6$) \cite[26]{Cech}, by \v{C}ech's monotonicity of dimension \cite[28]{Cech} \cite[3.1.20]{Engelking}, note $\dim(KU_x/K) \leqslant \dim(X/K) < \infty$.
So $KU_x$ is $K$-FTC (\ref{prop:FTC}).

Hence $KU_x$ has $K$-finite structure.
Since $KU_x$ is open in $X$, by Hanner's global-to-local principle \cite[III:7.9]{Hu}, each closed $H \leqslant K$ has $(KU_x)^H = KU_x \cap X^H \in \text{ANR}(\cM)$.
Therefore $KU_x \in K\text{-ANE}(K\text{-}\cM)$ by Jaworowski's criterion (\ref{lem:Jaworowski}).
Then $X = \bigcup_{x \in X} KU_x \in K\text{-ANE}(K\text{-}\cM)$ by Antonyan's open-union theorem \cite[5.7]{Antonyan2} as $K\text{-}\cM \subset K\text{-}\cP$ \cite[2.3]{Khan_linearLie}.
Thus $X \in G\text{-ANE}(G\text{-}\cM)$ by Abels' induction (\ref{lem:induction}).
\end{proof}

\begin{exm}
Consider the holomorph Lie group $G = T \rtimes_{\id} GL_2(\Z)$ and $X = G / U_1$, where the circle group $U_1$ is a subgroup of the maximal torus $T = U_1 \x U_1$ as first factor.
Then $X$ has a single $G$-orbit type but countably infinitely many $T$-orbit types by restriction, namely $\cO_T(X) = \{ \R(a,b)/\Z^2 ~|~ a>0 \text{ and } b \text{ coprime integers} \}$.
However, on compact sets $\cO_T$ is finite: $X$ has locally finitely many $T$-orbit types.
\end{exm}

\section{Finiteness of orbit types and of covering dimension}

In the first part of this section, we make an additional definition and a sequence of comparative observations that reflect more on concepts of the previous section.

\begin{defn}\label{defn:FUI}
Let $G$ be a topological group.
A Cartan $G$-space $X$ being \textbf{FUI} (\textbf{finite upper isotropies}) shall mean that $\bigcup \cO_G X$ has only finitely many conjugacy classes of upper bounds in the preordered set $(\cpt\,G, \preccurlyeq)$ of Definition~\ref{defn:partialorder}.
\end{defn}

\begin{exm}\label{exm:fot_FUI}
If $X$ has finitely many orbit types, then $X$ has the FUI property.
\end{exm}

\begin{exm}\label{exm:cpt_FUI}
Let $G$ be a locally compact Hausdorff group.
Let $X$ be a Palais $G$-space with $X$ Tikhonov and $X/G$ compact.
Any $G$-subset of $X$ is FUI using \ref{lem:approx_slice}.
\end{exm}

\begin{lem}\label{lem:FUI_cs}
Let $G$ be a topological group.
Let $X$ be a Cartan $G$-space satisfying the FUI property.
The union $\bigcup \cO_G(X)$ of orbit types is compactly supported (\ref{defn:compactlysupported}).
\end{lem}

\begin{proof}
Let $K_1, \ldots, K_n$ represent the upper bounds of $\bigcup \cO_G(X)$ in $(\cpt\,G,\preccurlyeq)$.
Note $C := \bigcup_{i=1}^n K_i$ is compact.
For all $x \in X$, then $G_x \preccurlyeq K_i \subseteq C$ for some $i$.
\end{proof}

Our contribution shall be to notice an equivalence with the maximal torus (\ref{rem:MaximalTorus}).

\begin{prop}\label{prop:linear_torus}
Let $L$ be any linear Lie group.
Let $T$ be a maximal torus in $L$.
Let $X$ be a Palais $L$-metrizable space satisfying FTC.
Then $X$ has only finitely many $L$-orbit types if and only if $X$ is FUI and has only finitely many $T$-orbit types.
\end{prop}

\begin{proof}
On the one hand, suppose that $X$ has only finitely many $L$-orbit types.
Then $X$ is FUI by Example~\ref{exm:fot_FUI}.
Since $X$ is $L$-FTC, it has $L$-finite structure.
So $X$ has $T$-finite structure by Lemma~\ref{lem:AAMV_restriction}, hence it has only finitely many $T$-orbit types.

Conversely, suppose $X$ is FUI and has only finitely many $T$-orbit types.
By Lemma~\ref{lem:FUI_cs} then Corollary~\ref{cor:transtoral}, we find $X$ has only finitely many $L$-orbit types.
\end{proof}

In the second part of this section, we state a folklore result from dimension theory (see after \cite[\foreignlanguage{russian}{Проблема}~2]{Dranishnikov}), which shall be documented by F\,D~Ancel \cite{Ancel}.
For locally compact spaces, $\dim_\Z$ is $\Z$-\v{C}ech cohomological dimension in the sense of Cohen \cite[2.10]{Cohen}; for normal spaces, $\dim$ is \v{C}ech--Lebesgue covering dimension.

\begin{lem}[Kozlowski]\label{lem:KA}
Let $X \in \text{ANR}(\cM)$ be a locally compact space.
Suppose $\dim_\Z(X)=n<\infty$.
Then $X$ has finite covering dimension.
Moreover, $\dim(X)=n$.
\end{lem}

Dranishnikov found $X \in \cM$ compact, $\dim_\Z(X)<\infty$, and $\dim(X)=\infty$ \cite{Dranishnikov1}.

\section{Equivariant topological manifolds}

The following generalization of \cite[3.1]{Khan_linearLie} now includes nonlinear cases of $G$.
It is an application of Smith theory, new $G$-ANR theory, and $G$-overhomotopy theory.

\begin{thm}\label{thm:main}
Let $G$ be any Lie group.
Let $M$ be a second-countable $\Z$-cohomology manifold \cite[1.3]{Khan_linearLie} equipped with a Palais $G$-action such that any nonempty $K$-skeleton ($K$-fixed set) $M^K = \{ x \in M ~|~ \forall g \in K : gx = x \}$ is locally contractible.
Then $M$ is $G$-homotopy equivalent to a countable proper $G$-CW complex.
Further, a $G$-map $f: M \longra N$ between such $\Z$-cohomology manifolds is a $G$-homotopy equivalence if and only if each map $f^K: M^K \longra N^K$ is a homotopy equivalence.
\end{thm}

\begin{proof}
Let $T$ denote a maximal torus in $G$ (see Remark~\ref{rem:MaximalTorus}).
By the Mann--Floyd theorem \cite[VI:1.1]{Borel_seminar}, any compact set in $M$ has only finitely many $T$-orbit types.
The definition of $\Z$-cohomology manifold \cite[I:3.3]{Borel_seminar} includes both that it is locally compact Hausdorff \cite[I:1.1]{Borel_seminar} (hence $T_{3.5}$) and it has finite $\Z$-\v{C}ech cohomological dimension \cite[2.10]{Cohen}.
Note $M$ is $G$-metrizable, as second-countable $T_3$ implies metrizable \cite[2]{Tikhonov}, and as both $G$ and $M$ are separable (\ref{lem:AntonyandeNeymet}).

Since $M$ has finite covering dimension (\ref{lem:KA}) and $M^K$ is closed in $M \in T_4$, note $\dim(M^K) \leqslant \dim(M) < \infty$ \cite[4]{Cech} (see also \cite[3.1.4]{Engelking}).
Then, by Kodama--Dugundji \cite[IV:7.1]{Hu}, $M^K$ is locally contractible if and only if $M^K \in \text{ANR}(\cM)$.
(The empty set is trivially an ANR.)
Thus $M \in G\text{-ANR}(G\text{-}\cM)$ by Theorem~\ref{thm:Jaworowski}.
So $M$ is $G$-homotopy equivalent to a proper $G$-CW complex \cite[1.1]{AE2}.
Since $M$ is separable, the `countable' is the same as in the second half of \cite[Proof~3.1]{Khan_linearLie}.

Lastly, our version (\ref{cor:JamesSegal}) of the James--Segal criterion is satisfied (\ref{rem:Stone}) if one works in the subclass $G\text{-}\cM \subset G\text{-}\cP^*$, since $M,N \in G\text{-ANR}(G\text{-}\cM)$.
Alternatively, and less directly because of more theoretical operations, one can use that $M$ and $N$ have the $G$-homotopy type of $G$-CW complexes, then the corresponding theorem for $G$-CW complexes \cite[II:2.7]{tomDieck_TG}, which is proven using $G$-obstruction theory.
\end{proof}

\begin{defn}\label{defn:Khan}
By \textbf{topological $G$-manifold} \cite[2.2]{Khan_compactLie}, we mean that the $H$-skeleton is a topological ($C^0$) manifold for each closed subgroup $H$ of a topological group $G$.
Any \emph{topological manifold} is separable, metrizable, and locally euclidean.
\end{defn}

Finally, we generalize \cite[Corollary~3.2]{Khan_linearLie} beyond the Lie group $G$ being linear.
Recall a $G$-CW complex is \emph{countable} if it has countably many $G$-cells \cite[1.4]{Matumoto_GCW}.

\begin{cor}\label{cor:main}
Let $G$ be an arbitrary Lie group.
Any topological $G$-manifold with Palais action has equivariant homotopy type of a countable proper $G$-CW complex.
Furthermore, a $G$-map between such spaces is a $G$-homotopy equivalence if and only if its restriction to their $K$-fixed sets is a homotopy equivalence for each $K \leqslant G$.
\end{cor}

\begin{proof}
Let $M$ be a topological $G$-manifold with Palais action.
Since each manifold $M^H$ is locally euclidean hence locally contractible, and since $M$ is a second-countable $\Z$-cohomology manifold, we obtain the conclusion by Theorem~\ref{thm:main}.
\end{proof}

Next, we generalize \cite[Corollary~3.5]{Khan_compactLie} from $\G$ virtually torsionfree and we moreover answer \cite[Footnote~5]{Khan_compactLie}, which asked if it is true for $\G$ residually finite.

\begin{exm}
Let $\G$ be a countable discrete group.
Any topological $\G$-manifold with properly discontinuous action has the $\G$-homotopy type of a countable $\G$-CW complex.
Hence, any $\G$-space $M$ with $M^H$ a contractible manifold if $H$ is finite and empty otherwise is a \emph{manifold model for $E_\fin\G$} in the sense of \cite{CDK1} \cite{CDK2}.
\end{exm}

Thus more tractible are its Davis--L\"uck $\Or\,\G$-spectral homology groups \cite[3.7, 4.3]{DL1}, since we conclude countability of the $\G$-CW complex that left-approximates.

In particular, we generalize Elfving's improved thesis \cite[Theorem~1]{Elfving2}.
The definition of \emph{locally linear}, along with some discussion, is found in \cite[3.6, 3.7]{Khan_compactLie}.
Note any smoothable action is locally linear, but not vice versa; see \cite[VI:9.6]{Bredon_TG}.

\begin{cor}[Elfving]
Let $G$ be any Lie group.
Any locally linear $G$-manifold with Palais action has the equivariant homotopy type of a $G$-CW complex.
\end{cor}

\begin{proof}
This special case now follows immediately from Corollary~\ref{cor:main}.
\end{proof}

\cite[\S4]{Khan_linearLie} has 4 uncountable families of $G$-manifolds that are not locally linear.


\subsection*{Acknowledgements}
I thank Christopher Connell for various basic discussions.
I am grateful to Ric Ancel and Alex Dranishnikov for email dialogue on Lemma~\ref{lem:KA}.
The referee kindly pointed out the special case Corollary~\ref{cor:ccc} is more recently known.

\bibliographystyle{alpha}
\bibliography{CountableApproximation_LieG.bib}

\end{document}